\documentclass[11pt]{amsart}
\usepackage{graphicx}

\usepackage{amsmath}
\usepackage{amsthm}
\usepackage{amssymb}
\usepackage{enumerate}
\usepackage{cleveref}

\newtheorem{lem}{Lemma}[section]

\newtheorem{prop}[lem]{Proposition}
\newtheorem{coro}[lem]{Corollary}
\newtheorem{thm}[lem]{Theorem}
\newtheorem{de}[lem]{Definition}

\newtheorem{pozn}[lem]{Remark}
\newtheorem{ex}[lem]{Example}

\newtheorem{observation}[lem]{Observation}
\def\pf{\begin{proof}}
\def\pfk{\end{proof}}

\crefname{thm}{theorem}{Theorems}
\crefname{thrm}{theorem}{Theorems}
\crefname{coro}{corollary}{Corollaries}
\crefname{ex}{example}{Examples}
\crefname{lem}{lemma}{lemmas}
\crefname{lmm}{lemma}{lemmas}
\crefname{claim}{claim}{Claims}
\crefname{klejm}{claim}{Claims}
\crefname{obs}{observation}{Observations}
\crefname{proposition}{proposition}{Propositions}
\crefname{prop}{proposition}{Propositions}
\crefname{de}{definition}{Definitions}
\crefname{pozn}{remark}{remarks}

\begin{document}
\title{Infinite words with finite defect  }

\author[Balkov\'a]{L{\!'}ubom\'{\i}ra~Balkov\'a}
\address[Balkov\'a]{Department of Mathematics, FNSPE Czech Technical University in
Prague, Trojanova 13, 120 00 Praha 2, Czech Republic}
\email{lubomira.balkova@fjfi.cvut.cz}

\author[Pelantov\'a]{Edita~Pelantov\'a}
\address[Pelantov\'a]{Department of Mathematics, FNSPE Czech Technical University in
Prague, Trojanova 13, 120 00 Praha 2, Czech Republic}
\email{edita.pelantova@fjfi.cvut.cz}

\author[Starosta]{\v St\v ep\'an Starosta}
\address[Starosta]{Department of Mathematics, FNSPE Czech Technical University in
Prague, Trojanova 13, 120 00 Praha 2, Czech Republic}
\email{staroste@fjfi.cvut.cz}
\keywords{defect, rich words, palindrome, palindromic complexity, factor complexity, morphisms of class $P$}
\subjclass[2010]{68R15}
\date{\today}

\begin{abstract}
In this paper, we provide a~new characterization of uniformly
recurrent words with finite defect based on a~relation between the
palindromic and factor complexity. Furthermore, we introduce a~class of morphisms $P_{ret}$ closed under composition and we show that a~uniformly recurrent word
with finite defect is an image of a~rich (also called full) word under a morphism
of class $P_{ret}$. This class is closely related to the well-known class $P$ defined by Hof, Knill,
and Simon; every morphism from $P_{ret}$ is conjugate to a~morphism of class $P$.
\end{abstract}
\maketitle

%%%%%%%%%%%%%%%%%%%%%%%%%%%%%%%%%%%%%%%%%%%%%%%%%%%%%%%%%%%%%%%%%%%%%%%%%%%%%%%%%%%%%%%%%%%%%%%%%%%%%%%%%%%%%%%%%%
%%%%%%%%%%%%%%%%%%%%%%%%%%%%%%%%%%%%%%%%%%%%%%%%%%%%%%%%%%%%%%%%%%%%%%%%%%%%%%%%%%%%%%%%%%%%%%%%%%%%%%%%%%%%%%%%%%
%%%%%%%%%%%%%%%%%%%%%%%%%%%%%%%%%%%%%%%%%%%%%%%%%%%%%%%%%%%%%%%%%%%%%%%%%%%%%%%%%%%%%%%%%%%%%%%%%%%%%%%%%%%%%%%%%%
\section{Introduction}

The upper bound $|w|+1 $ on the number of palindromes occurring in
a finite word $w$ given by X.~Droubay, J.~Justin,  and  G.~Pirillo
in \cite{DrJuPi} initiated  many interesting investigations on
palindromes in  infinite words as well. An infinite word for which
the upper bound is attained for any of its factors is called rich
or full. There exist several characterizations of rich words based
on the notion of complete return words~\cite{GlJuWiZa}, on the
longest palindromic suffix and prefix of a factor~\cite{DrJuPi,
BuLuGlZa2}, on the palindromic and factor
complexity~\cite{BuLuGlZa} and most recently on the bilateral
orders of factors~\cite{BaPeSta}. Brlek et al. suggested in
\cite{BrHaNiRe} to study the defect of a finite word $w$ defined
as the difference between the upper bound $|w|+1 $ and the actual
number of palindromes contained in $w$. The defect of an infinite
word is then defined as the maximal defect of a factor of the
infinite word. In this convention, rich words are precisely the
words with zero defect. In this paper we focus on uniformly
recurrent words with finite defect. Let us point out that periodic
words with finite defect have been already described
in~\cite{BrHaNiRe} and in~\cite{GlJuWiZa}. In Section 2 we
introduce notation and summarize known results on rich words and
words with finite defect. In Section 3 the notion of oddities and
the characterization of uniformly recurrent words with finite
defect based on oddities from~\cite{GlJuWiZa} is recalled and, as
an immediate consequence, two more useful characterizations are
deduced. The main result is a~new characterization of uniformly
recurrent words with finite defect based on a~relation between the
palindromic and factor complexity, see Theorem~\ref{PALaDefect} in
Section 4. Furthermore, we introduce a~class of morphisms $P_{ret}$ closed under composition of morphisms and we show that a~uniformly recurrent word
with finite defect is an image of a~rich word under a morphism
of class $P_{ret}$, see Theorem~\ref{obraz} in Section
5. This class is closely related to the well-known class $P$ defined by Hof, Knill, and Simon in~\cite{HoKnSi}; every morphism from $P_{ret}$ is conjugate to a~morphism of class $P$.

%%%%%%%%%%%%%%%%%%%%%%%%%%%%%%%%%%%%%%%%%%%%%%%%%%%%%%%%%%%%%%%%%%%%%%%%%%%%%%%%%%%%%%%%%%%%%%%%%%%%%%%%%%%%%%%%%%
%%%%%%%%%%%%%%%%%%%%%%%%%%%%%%%%%%%%%%%%%%%%%%%%%%%%%%%%%%%%%%%%%%%%%%%%%%%%%%%%%%%%%%%%%%%%%%%%%%%%%%%%%%%%%%%%%%
%%%%%%%%%%%%%%%%%%%%%%%%%%%%%%%%%%%%%%%%%%%%%%%%%%%%%%%%%%%%%%%%%%%%%%%%%%%%%%%%%%%%%%%%%%%%%%%%%%%%%%%%%%%%%%%%%%
\section{Preliminaries}

By $\mathcal{A}$ we denote a~finite set of symbols, usually called
{\em letters}; the set $\mathcal{A}$ is therefore called an {\em
alphabet}. A finite string $w=w_0w_1\ldots w_{n-1}$ of letters of
$\mathcal{A}$ is said to be a~{\em finite word}, its length is
denoted by $|w| = n$. Finite words over  $\mathcal{A}$ together with
the operation of concatenation and the empty word $\epsilon$ as the
neutral element form a~free monoid $\mathcal{A}^*$. The map
$$w=w_0w_1\ldots w_{n-1} \quad \mapsto \quad \overline{w} =
w_{n-1}w_{n-2}\ldots w_{0}$$ is a bijection on $\mathcal{A}^*$, the
word $\overline{w}$ is called the {\em reversal} or the {\em mirror
image} of $w$. A~word $w$ which coincides with its mirror image is
a~{\em palindrome}.

Under an {\em infinite word} we understand an infinite string
${\mathbf u}=u_0u_1u_2\ldots $ of letters from $\mathcal{A}$.  A~finite word
$w$ is a~{\em factor} of a~word $v$ (finite or infinite) if there
exist words $p$ and $s$ such that $v= pws$.
If $p = \epsilon$, then $w$ is said to be a~{\em prefix} of
$v$, if $s = \epsilon$, then $w$ is a~{\em suffix} of~$v$.

The {\em language} $\mathcal{L}({\mathbf u})$ of an infinite word ${\mathbf u}$ is the
set of all its factors. Factors of $\mathbf u$ of length $n$ form the set
denoted by $\mathcal{L}_n({\mathbf u})$. Clearly, $\mathcal{L}({\mathbf u})=\cup_{n\in
\mathbb{N}}\mathcal{L}_n({\mathbf u})$. {We say that the language
$\mathcal{L}({\mathbf u})$ is {\em closed under reversal} if $\mathcal{L}({\mathbf u})$
contains with every factor $w$ also its reversal $\overline{w}$.}

For any factor $w\in \mathcal{L}({\mathbf u})$, there exists an index $i$
such that $w$ is a prefix  of  the infinite word
$u_iu_{i+1}u_{i+2} \ldots$. Such an index is called an {\em
occurrence} of $w$ in ${\mathbf u}$. If each factor of $\mathbf u$ has infinitely many
occurrences in ${\mathbf u}$, the infinite word $\mathbf u$ is said to be {\em
recurrent}. It is easy to see that if the language of ${\mathbf u}$ is
closed under reversal, then ${\mathbf u}$ is recurrent (a~proof can be found in~\cite{GlJuWiZa}). For a~recurrent
infinite word ${\mathbf u}$, we may define the notion of a~{\em complete
return word} of any $w \in\mathcal{L}({\mathbf u})$. It is a~factor $v\in
\mathcal{L}({\mathbf u})$ such that $w$ is a prefix and a suffix of $v$ and
$w$ occurs in $v$ exactly twice. Under a~{\em
return word} of a factor $w$ is usually meant
a word $q \in \mathcal{L}({\mathbf u})$  such that $qw$ is a complete return word of $w$.
If any factor $w \in \mathcal{L}({\mathbf u})$ has only finitely many return words, then
the infinite word ${\mathbf u}$ is called {\em uniformly recurrent}. If ${\mathbf u}$ is a~uniformly
recurrent word, we can assign to any $n \in \mathbb{N}$ the minimal number
$R_{{\mathbf u}}(n) \in  \mathbb{N}$ such that we have for any $v  \in
\mathcal{L}({\mathbf u}) \ \hbox {with} \ |v| \geq  R_{{\mathbf
u}}(n)$
$$  \{ w \mid |w| = n,  \  w \ \hbox{ is a factor of } \ v\}  =
\mathcal{L}_n({\mathbf u}),
$$
or equivalently, any piece of ${\mathbf u}$ which is longer than or equal to
$R_{{\mathbf u}}(n)$ contains already all factors of ${\mathbf u}$ of
length $n$.  The map $n \to R_{{\mathbf u}}(n)$ is usually
called the {\em recurrence function} of ${\mathbf u}$.
In particular, any fixed point of a~primitive morphism is uniformly recurrent,
where a~morphism $\varphi$ over an alphabet $\mathcal A$ is {\em primitive}
if there exists an integer $k$ such that for every $a \in \mathcal A$
the $k$-th iteration $\varphi^k(a)$ contains all letters of $\mathcal A$.

The {\em factor complexity} of an infinite word ${\mathbf u}$ is a~map
$\mathcal{C}: \mathbb{N} \mapsto \mathbb{N}$ defined by the
prescription $\mathcal{C}(n):=\# \mathcal{L}_n({\mathbf u})$. To determine the
first difference of the factor complexity,  one has to count the possible
{extensions} of factors of length $n$. A~{\em right extension} of $w
\in \mathcal{L}({\mathbf u})$ is any letter $a\in \mathcal{A}$ such that $w
a\in \mathcal{L}({\mathbf u})$. Of course, any factor of ${\mathbf u}$ has at least one
right extension. A~factor $w$ is called {\em right special} if $w$
has at least two right extensions. Similarly, one can define a~{\em
left extension} and a~{\em left special} factor. We will deal only
with recurrent infinite words ${\mathbf u}$. In this case, any factor of $\mathbf u$
has at least one left extension. We say that $w$ is a~{\em
bispecial} factor if it is right and left special.
%%%%%%%%%%%%%%%%%%%%%%%%%%%%%%%%%%%%%%%%%%%%%%%%%%%%%%%%%%%

In our article we focus on words in some sense opulent in palindromes,
therefore we  will introduce several notions connected with palindromic factors.

\medskip

The {{\em defect} $D(w)$ of a~finite word $w$ is the difference between
the utmost number of palindromes $|w|+1$ and the actual number of
palindromes contained in $w$. Finite words with zero defects -- called {\em rich} words -- can be viewed as the most saturated by
palindromes. This definition may be extended to infinite words as
follows.}

\begin{de} An infinite word ${\mathbf u} = u_0u_1u_2 \ldots$ is called
{\it rich}, if for any index $n\in \mathbb{N}$  the prefix
$u_0u_1u_2 \ldots u_{n-1}$ of length $n$ contains exactly $n+1$
different palindromes.
\end{de}
\begin{pozn}
We keep here the terminology introduced by Glen et al. in~\cite{GlJuWiZa} in 2007, which seems to us to be prevalent nowadays. However, Brlek et al.
in~\cite{BrHaNiRe} baptized such words full already in 2004.
\end{pozn}
Let us remark that not only all prefixes of rich words are rich, but also all factors
are rich. A result  from~\cite{DrJuPi} will provide us
with a~handful tool which helps to evaluate the defect of a factor.

\begin{prop}[\cite{DrJuPi}]\label{lps} A finite or infinite word $\mathbf u$ is rich if and only if
the longest palindromic suffix of $w$ occurs
exactly once in $w$ for any prefix $w$ of $\mathbf u$.
\end{prop}
The longest palindromic suffix of a factor $w$ will occur often
 in our considerations, therefore we will denote it
 by $lps(w)$. In accordance  with the terminology introduced in
\cite{DrJuPi}, the factor with a~unique occurrence in another factor
is called {\em unioccurrent}. From the proof of the previous proposition
directly follows the next corollary.

\begin{coro}\label{kolikunioccurent} The defect $D(w)$ of a~finite word $w$ is equal to the
number of prefixes $w'$ of $w$, for which the longest palindromic
suffix of $w'$ is not unioccurrent in $w'$.
\end{coro}
This corollary implies that $D(v) \geq D(w)$ whenever $w$ is
a~factor of $v$. It enables to give a~reasonable definition of
the defect of an infinite word (see~\cite{BrHaNiRe}).

\begin{de}  The defect of  an  infinite word  $\mathbf{u}$ is the  number (finite or infinite)
$$D({\mathbf{u}}) = \sup \{ D(w) \bigm | w \ \text{is a~prefix of $\mathbf u$}\}\,.$$
\end{de}
Let us point out several facts concerning defects that are easy to prove:
\begin{enumerate}
\item
If we consider all factors of a~finite or an infinite word $\mathbf u$,
we obtain the same defect, i.e.,
$$D({\mathbf{u}}) = \sup \{ D(w) \bigm | w \in \mathcal{L}(\mathbf{u})\} \,.$$
\item Any infinite word with finite defect contains infinitely many palindromes.
\item
Infinite words with zero defect correspond exactly to rich words.
\end{enumerate}
Periodic words with finite defect have been studied in~\cite{BrHaNiRe} and in~\cite{GlJuWiZa}.
It holds that the defect of an infinite periodic word with the minimal period $w$ is finite if and only if $w=pq$, where both $p$ and $q$ are palindromes.
In \cite{GlJuWiZa} words with finite defect have been baptized {\em almost rich} and the richness of a~word was described using complete return words.
\begin{prop}[\cite{GlJuWiZa}]\label{returnrich} An infinite word ${\mathbf u}$ is rich if and only if
all complete return words of any palindrome are palindromes.
\end{prop}

The authors of~\cite{DrJuPi} who were the first ones to tackle
this problem showed that Sturmian and episturmian words are rich.
In~\cite{BrHaNiRe}, an insight into the richness of periodic words
can be found.

\medskip

The number of  palindromes of a~fixed length occurring in an
infinite word is measured by the so called {\it palindromic
complexity} ${\mathcal{P}}$, a~map which assigns to any non-negative
integer $n$ the number
$$ {\mathcal{P}}(n) := \#\{ w \in {\mathcal{L}_n}(u)\mid w \ \
\hbox{is a palindrome}\}\,.$$
The palindromic complexity is bounded
by the first difference of factor complexity. The following proposition is proven in
\cite{BaMaPe} for uniformly recurrent words, however the uniform recurrence is not needed in the proofs,
thus it holds for any infinite words with language closed under reversal.

\begin{prop}[\cite{BaMaPe}]\label{BaMaPe}
 Let $\mathbf u$ be an infinite word
with language closed under reversal. Then
\begin{equation}\label{nerovnost}
{\mathcal {P}}(n) + {\mathcal {P}}(n+1) \leq {\mathcal {C}}(n+1) -
{\mathcal {C}}(n) +2\,,
\end{equation}
for all $n\in \mathbb{N}$.
\end{prop}

\medskip

It is shown in \cite{BuLuGlZa} that this bound can be used for
the characterization of rich words as well.
The following proposition states this fact.

\begin{prop}[\cite{BuLuGlZa}]\label{rich_pal} An infinite word  ${\mathbf u}$
with language closed under reversal is rich if and only if the
equality in \eqref{nerovnost} holds for all $n\in \mathbb{N}$.
\end{prop}

The most recent characterization of rich words given in \cite{BaPeSta}
exploits the notion of the bilateral order $b(w)$  of a factor and
the palindromic extension of a~factor. The  bilateral order was
introduced in \cite{Ca} as ${\rm b}(w) = \#\{ awb \mid awb\in
\mathcal{L}({\mathbf u}), a,b \in \mathcal{A} \} - \#\{ aw  \mid aw\in  \mathcal{L}({\mathbf u}), a \in \mathcal{A}
\}- \#\{ wb  \mid wb\in  \mathcal{L}({\mathbf u}), b \in \mathcal{A} \} +1\,.$ The
set of palindromic extensions of  a  palindrome $w \in  \mathcal{L}({\mathbf u})$
is defined by $ {\rm Pext}(w) =\{ awa \mid  awa \in \mathcal{L}({\mathbf
u}), a \in \mathcal{A} \} $.
\begin{prop}[\cite{BaPeSta}] \label{bw_a_pext} An infinite word  ${\mathbf u}$
with language closed under reversal is rich if and only if any
bispecial factor $w$ satisfies:
\begin{itemize}
\item if $w$ is non-palindromic, then ${\rm b}(w)=0$,
 \item if $w$
is  a palindrome, then ${\rm b}(w)=\#{\rm Pext}(w) -1. $
\end{itemize}
\end{prop}

%%%%%%%%%%%%%%%%%%%%%%%%%%%%%%%%%%%%%%%%%%%%%%%%%%%%%%%%%%%%%%%%%%%%%%%%%%%%%%%%%%%%%%%%%%%%%%%%%%%%%%%%%%%%%%%%%%
%%%%%%%%%%%%%%%%%%%%%%%%%%%%%%%%%%%%%%%%%%%%%%%%%%%%%%%%%%%%%%%%%%%%%%%%%%%%%%%%%%%%%%%%%%%%%%%%%%%%%%%%%%%%%%%%%%
%%%%%%%%%%%%%%%%%%%%%%%%%%%%%%%%%%%%%%%%%%%%%%%%%%%%%%%%%%%%%%%%%%%%%%%%%%%%%%%%%%%%%%%%%%%%%%%%%%%%%%%%%%%%%%%%%%
\section{Characterizations of words with finite defect}

Uniformly recurrent words with finite defect are characterized using the notion of oddities in Proposition 4.8 from~\cite{GlJuWiZa}.
It is based on the following lower bound.
\begin{prop}[Proposition 4.6~\cite{GlJuWiZa}] For any infinite word $\mathbf u$ it holds
$$D({\mathbf u}) \geq \#\bigl\{ \{v,\overline{v}\} \bigm | v\neq \overline{v}\
 \hbox{and}\  v \ \text{or}\ \overline{v}\  \hbox{is a~complete return word in ${\mathbf u}$ of a~palindrome}  \ w\bigr\}.$$
\end{prop}
The set $\{v, \overline{v}\}$ is called an {\em oddity}.
It is clear that for uniformly recurrent words with a~finite number of distinct palindromes,
the defect is infinite, however the number of oddities is finite. Moreover,
even for uniformly recurrent words with infinitely many palindromes, it can hold
$$D({\mathbf u}) > \#\bigl\{ \{v,\overline{v}\} \bigm | v\neq \overline{v}\
 \hbox{and}\  v \ \text{or}\ \overline{v}\  \hbox{is a~complete return word in ${\mathbf u}$ of a~palindrome}  \ w\bigr\}.$$
We take an example for this situation from~\cite{GlJuWiZa}.
\begin{ex}
Let ${\mathbf u}=(abcabcacbacb)^{\omega}$, where $\omega$ denotes an infinite repetition, then $D({\mathbf u})=4$, but the number of oddities is equal to~$3$.
\end{ex}
However, the defect of an aperiodic word can also exceed the number of oddities. For instance, if we replace in Example~\ref{finite_defect} the substitution $\sigma$ with $0\to cabcabcbacbac, \ 1 \to d$, then it is easy to show that $D(\mathbf u)=4$, but the number of oddities is 3.

We can now recall the characterization of words with finite defect based on oddities.
\begin{prop}[Proposition 4.8~\cite{GlJuWiZa}]\label{oddities_defect}
A~uniformly recurrent word $\mathbf u$ has infinitely many oddities if and only if
$\mathbf u$ contains infinitely many palindromes and $D({\mathbf u})=\infty$.
\end{prop}

As an immediate consequence of Proposition~\ref{oddities_defect}, we obtain the following characterizations of
infinite words with finite defect.
\begin{thm}\label{vsechno} Let ${\mathbf u}$ be a uniformly recurrent word containing infinitely many palindromes.
Then the following statements are equivalent:
\begin{enumerate}
\item[1.] $D({\mathbf u}) <  \infty $,

\item[2.] $\mathbf u$ has a~finite number of oddities,

\item[3.] there exists an integer $K$ such that all complete return words of any palindrome from $
\mathcal{L}({\mathbf u})$  of length at least $K$  are palindromes,

\item[4.] there exists an integer $H$ such that for any prefix $f$ of ${\mathbf u}$  with $|f| \geq H$
the longest palindromic  suffix  of $f$
 is unioccurrent in $f$.
\end{enumerate}
\end{thm}
\pf
{\em 1.} and {\em 2.} are equivalent by Proposition~\ref{oddities_defect}.
It follows directly from the definition of oddities that {\em 2.} and {\em 3.} are equivalent.
Corollary~\ref{kolikunioccurent} implies that {\em 1.} and {\em 4.} are equivalent.

\pfk
It is easy to see that the last statement of Theorem~\ref{vsechno} can be equivalently rewritten as:
There exists an integer $H$ such that for any factor $f$ of ${\mathbf u}$  with $|f| \geq H$
the longest palindromic  suffix  of $f$ is unioccurrent in $f$.

Let us stress that if we put in the previous theorem  $D({\mathbf u})
= K=H=0$, the points {\em 1.}, {\em 3.}, and {\em 4.} become known results on rich words, see Propositions~\ref{returnrich} and~\ref{lps}.

\begin{ex}\label{finite_defect}
Let us provide an example of a~uniformly recurrent word $\mathbf u$ with finite defect and let us find for
$\mathbf u$ the lowest values of constants $K$ and $H$ from Theorem~\ref{vsechno}.
Take the Fibonacci word $\mathbf v$, i.e., the fixed point of $\varphi: 0 \to 01, \ 1 \to 0$.
Define $\mathbf u$ as its morphic image $\sigma(\mathbf v)$, where $\sigma: 0 \to cabcbac, \ 1 \to d$.

It is easy to show that all palindromes of length
greater than $1$ and the palindromes $a$, $b$, and $d$ have only
palindromic complete return words. Hint: long palindromes in $\mathbf
u$ contains in their center images of non-empty palindromes from
$\mathbf v$ that have palindromic complete return words by the richness of
$\mathbf v$. The only non-palindromic complete return of $c$ is $cabc$, thus there is exactly one oddity $\{cabc, cbac\}$.
In order to show that $D({\mathbf
u})=1$, it suffices to verify that no prefixes longer than $cabc$ have $c$ as their
longest palindromic suffix. This follows directly from the form of
$\sigma$. The lowest values of the constants $K$ and $H$ are: $K=2$, $H=5$.
\end{ex}
%%%%%%%%%%%%%%%%%%%%%%%%%%%%%%%%%%%%%%%%%%%%%%%%%%%%%%%%%%%%%%%%%%%%%%%%%%%%%%%%%%%%%%%%%%%%%%%%%%%%%%%%%%%%%%%%%%
%%%%%%%%%%%%%%%%%%%%%%%%%%%%%%%%%%%%%%%%%%%%%%%%%%%%%%%%%%%%%%%%%%%%%%%%%%%%%%%%%%%%%%%%%%%%%%%%%%%%%%%%%%%%%%%%%%
%%%%%%%%%%%%%%%%%%%%%%%%%%%%%%%%%%%%%%%%%%%%%%%%%%%%%%%%%%%%%%%%%%%%%%%%%%%%%%%%%%%%%%%%%%%%%%%%%%%%%%%%%%%%%%%%%%
\section{Palindromic complexity of words with finite defect}
The aim of this section is to prove the following new characterization of infinite words with finite defect
based on a~relation between the palindromic and factor complexity.
\begin{thm}\label{PALaDefect}
Let ${\mathbf u}$ be a~uniformly recurrent word.
Then $D({\mathbf u}) < \infty $ if and only if there exists an integer $N$
such that $$
{\mathcal {P}}(n) + {\mathcal {P}}(n+1) = {\mathcal {C}}(n+1) -
{\mathcal {C}}(n) +2$$
holds for all $n \geq N$.
\end{thm}
Notice that if we set $N=0$ in the previous theorem, then we
obtain the known characterization of rich words from
Proposition~\ref{rich_pal} (which holds even under a~weaker assumption that
${\mathcal L}(\mathbf u)$ is closed under reversal).

In the sequel, we will prove two propositions that together with the equivalent
characterizations of words with finite defect from Theorem~\ref{vsechno} imply
Theorem~\ref{PALaDefect}.
As we have already mentioned, all words with language
closed under reversal satisfy the inequality in
Proposition~\ref{BaMaPe}. A~direct consequence of its  proof
 given in \cite{BaMaPe} is a~necessary and
sufficient condition for the equality in~\eqref{nerovnost}. To
formulate this condition in Lemma \ref{tree}, we introduce two
auxiliary notions.

Let $\mathbf u$ be an infinite word with language closed under reversal
and let $n$ be a given positive integer.

An {\it $n$-simple path} $e$ is a factor of $\mathbf u$ of
length at least $n+1$ such that the only special (right or left)
factors of length $n$ occurring in $e$ are its prefix and  suffix
of length $n$. If $w$ is the prefix of $e$ of length $n$ and $v$
is the suffix of $e$ of length $n$, we say that the $n$-simple
path $e$ starts in $w$ and ends in $v$.

We will denote by  $G_n$ an undirected graph whose set
of vertices is formed by unordered pairs $(w,\overline{w})$ such
that $w\in {\mathcal{L}_n}(u)$ is right or left special. We
connect two
 vertices $(w,\overline{w})$ and $(v,\overline{v})$  by
an unordered
 pair $(e,\overline{e})$ if $e$ or $\overline{e}$ is an $n$-simple path
 starting in $w$ or $\overline{w}$ and ending in $v$ or
 $\overline{v}$.

\medskip

\noindent Note that the graph $G_n$ may have multiple edges and
loops.

\medskip

\begin{lem}\label{tree}
Let $\mathbf u$ be an infinite word with language closed under reversal. The equality in \eqref{nerovnost} holds for an integer $n \in \mathbb N$
if and only if both of the
following conditions are met:
 \begin{enumerate}
\item[1.]\label{prvni}  The graph
 $G_n$ after removing loops is a tree.
 \item[2.]\label{druha}  Any $n$-simple path forming a~loop in the graph $G_n$ is a palindrome.
 \end{enumerate}
\end{lem}

\begin{prop}\label{Pal_K} Let ${\mathbf u}$ be an infinite word with language
 closed under reversal. Suppose that there exists an integer $N$  such that for all $n\geq N$ the equality
 ${\mathcal {P}}(n) + {\mathcal {P}}(n+1)= {\mathcal {C}}(n+1) -
{\mathcal {C}}(n) +2$  holds. Then the complete return words of any
palindromic factor of length $n \geq N$ are palindromes.
\end{prop}
\pf  Assume the contrary: Let $p= p_1p_2 \ldots p_k$  be a~palindrome  with  $k \geq N$ and let $v$ be its complete return
word which is not a~palindrome. Clearly $|v|>2|p|$. Then there exist a~factor $f$
(possibly empty) and two different letters $x$ and $y$ such that
$v=pfxv'y\overline{f}p$.

Let us consider the graph $G_n$, where $n$ is the length of the factor
$w:=pf$, i.e., $n \geq N$. Since the language of ${\mathbf u} $ is closed under
reversal, the factor $w$ is right special - the letters $x$ and
$y$ belong to its right extensions.

If the complete return word $v$ contains no other right or left
special factors, then the non-palindromic $v$ is an $n$-simple
path which starts in $w=pf$ and ends in
$\overline{w}=\overline{f}p$ - a~contradiction with the
condition~{\em 2.} in Lemma~\ref{tree}.

Let  $v$ contain other left or right special factors of length
$n$. We find the prefix of $v$ which is an $n$-simple path. This simple
path starts in $w$, its ending point is a special factor, we
denote it by $A$. Since $v$ is a~complete return word of $p$, we
have $A\neq w,\overline{w}$. So in the graph $G_n$, the vertices
$ (w,\overline{w})$ and $(A,\overline{A})$ are connected with an
edge. Similarly, we find the suffix of $v$ which is an $n$-simple path
and we denote its starting point by $B$, its ending point is
$\overline{w}$. Again, $B\neq w,\overline{w}$ and the vertices $
(w,\overline{w})$ and $(B,\overline{B})$ are connected with an
edge. So in $G_n$ we have a~path with two edges which connects
$(A,\overline{A})$ and $(B,\overline{ B})$  and the vertex $
(w,\overline{w})$ is its intermediate vertex.

\medskip

The special factors $A$ and $B$ are factors of $p_2\ldots
p_kfxv'y\overline{f}p_k\ldots p_2$,
it means that in the graph $G_n$  there exists a walk, and
therefore a path\footnote{Along a walk vertices may occur with repetition, in a path any vertex appears at most once.} as well,
between the vertices $(A,\overline{A})$ and $(B,\overline{B})$ which
does not use the vertex $(w,\overline{w})$.

\medskip
Finally, if $(A,\overline{A})$ and $(B,\overline{B})$ coincide,
then we have in $G_n$ a multiple edge between $(A, \overline{A})$ and $(w, \overline{w})$. If $(A,\overline{A})\neq
(B,\overline{B})$, then  in $G_n$ we have  two different paths
connecting $(A,\overline{A})$ and $(B,\overline{ B})$. Together,
$G_n$ is not a tree after removing loops - a~contradiction  with the condition~{\em
1.} in Lemma~\ref{tree}.

\pfk

\begin{lem}\label{alternace} Let ${\mathbf u}$ be an infinite word whose language is closed under reversal.
Let ${\mathbf u}$ have the following property:
there exists an integer $H$ such that for any factor $f \in \mathcal{ L}({\mathbf u})$  with $|f| \geq H$
the longest palindromic  suffix  of $f$ is unioccurrent in $f$.
Let $w$ be a non-palindromic factor of ${\mathbf u}$ with $|w|\geq H$
and $v$ be a palindromic factor of ${\mathbf u}$ with $|v|\geq H$.
 Then
\begin{itemize}
\item  occurrences of $w$ and $\overline{w}$ in ${\mathbf u}$
alternate, i.e., any complete return word of $w$ contains the
factor $\overline{w}$,

\item any factor $e$ of ${\mathbf u}$ with a prefix $w$ and a suffix
$\overline{w}$, which has no other occurrences of $w$ and
$\overline{w}$, is a~palindrome,

\item any complete return word of $v$ is a~palindrome.

\end{itemize}
\end{lem}

\pf Consider a~non-palindromic factor $w$ such that $|w|\geq H$.  Let $f$ be a complete
return word of $w$. Since $|w| \geq H$, its complete return word satisfies
$|f|\geq H$. According to the assumption, $lps(f)$, the longest palindromic suffix of $f$, is
unioccurrent in $f$. Its length satisfies necessarily $|lps(f)|
> |w|$ - otherwise a contradiction with the unioccurrence of
$lps(f)$. Clearly, the palindrome   $lps(f)$ has a suffix $w$ and
thus a prefix $\overline{w}$, i.e., the complete return word $f$ of
$w$ contains $\overline{w}$ as well. Moreover, we have proven
that any factor $e$, which has a~prefix $\overline{w}$ and a~suffix $w$ and which has no other occurrences of $w$ and
$\overline{w}$, is the longest palindromic suffix of a~complete return word
of $w$, therefore $e=lps(f)$, i.e., the factor $e$ is a~palindrome.

Consider a~palindromic factor $v$, its complete return word $f$
and the longest palindromic suffix of $f$. Since $v$ is a~palindromic suffix of $f$,
necessarily  $|lps(f)|\geq |v|$.  As $|v|
\geq H$,  $lps(f)$ is unioccurrent in $f$. Hence, $|lps(f)|>|v|$. If $lps(f)$ is shorter than the whole $f$, then
the complete return word $f$ contains at least three occurrences
of $w$ - a contradiction. Thus, $lps(f)=f$, i.e., $f$ is a~palindrome.

\pfk

\begin{prop}\label{H_Pal} Let ${\mathbf u}$ be an infinite word whose language is closed under reversal.  Let ${\mathbf u}$ have the following property:
 there exists an  integer $H$ such that for any factor $f \in \mathcal{ L}({\mathbf u})$  with $|f| \geq H$
the longest palindromic  suffix  of $f$
 is unioccurrent in $f$. Then  $$2+ \mathcal{C}(n+1)-\mathcal{C}(n) = \mathcal{P}(n+1)+
 \mathcal{P}(n)\quad  \hbox{
  for any } \ n \geq H\,.$$
\end{prop}

\pf

We have to show that both conditions of Lemma~\ref{tree} are
satisfied for any $n \geq H$.\\
The condition~{\em 1.}: Let $(w,\overline{w})$ and $ (v,\overline{v})$ be two distinct vertices in
 the graph $G_n$, where  $n\geq H$. We say that an unordered  couple $ (f,\overline{f})$
  is a~{\em realization} of a path  between  these  two  vertices
if
\begin{itemize}
\item   either the factor $f$ or the factor $\overline{f}$ has the
property: $w$ or $\overline{w}$ is its prefix and $v$ or
$\overline{v}$ is its suffix,

\item there exist indices  $i, \ell \in \mathbb{N}, i< \ell$  such
that either the factor  $f$  or  the factor $\overline{f}$
coincides with the factor $u_iu_{i+1}\ldots u_\ell$ and factors
$w$, $\overline{w}$, $v$, and $\overline{v}$ do not occur in
$u_{i+1}\ldots u_{\ell-1}$.
\end{itemize}
The number  $i$ is called an {\em index} of the realization $(f,\overline{f})$.

Since ${\mathbf u}$ is recurrent, there exists at least one
realization  for any pair of vertices $(w,\overline{w})$ and $
(v,\overline{v})$ and  any realization has infinitely many
indices. Consider a realization $ (f,\overline{f})$ and its index~$i$.
WLOG $f=u_iu_{i+1} \ldots u_{\ell}$ and $w$ is a
prefix of $f$ and $v$ a suffix of $f$. Since ${\mathbf u}$ is
recurrent, we can find the smallest index $m>\ell$ such that $u' =
u_iu_{i+1} \ldots u_{\ell}\ldots u_{m}$ has a suffix
$\overline{w}$. According to Lemma  \ref{alternace}, $u'$ is a
palindrome. Therefore its suffix of length $|f|$ is exactly
$\overline{f}$. This means that the index $m-|f|+1$  is an
index of the same realization of a path between
$(w,\overline{w})$ and $ (v,\overline{v})$.  As the factor
$u_{i+1}\ldots u_{m-1}$ does not contain neither the factor $w$ nor
$\overline{w}$, no index $j$ strictly between $i$  and $m-|f|+1$
is an index of any realization of a path between
$(w,\overline{w})$ and $ (v,\overline{v})$.

 We have shown that between any pair of two consecutive indices of
 one specific
 realization  $(f, \overline{f})$ of a path between $(w,\overline{w})$ and $
 (v,\overline{v})$ there does not exist any index of any other realization  $(g,\overline{g})$
 of a path between $(w,\overline{w})$ and $
 (v,\overline{v})$. This means that there exists a~unique
 realization of a path between $(w,\overline{w})$ and $
 (v,\overline{v})$, which implies  that in the graph $G_n$ there
 exists a~unique path between vertices $(w,\overline{w})$ and $
 (v,\overline{v})$. Since this is true for all pairs of vertices of
 $G_n$, the graph $G_n$ after removing loops is a tree.\\

The condition~{\em 2.}: Let $w\in \mathcal{L}({\mathbf u})$ be a special factor (palindromic
or non-palindromic) with $|w| = n \geq H$. An $n$-simple path $f$
starting in $w$ and ending in $\overline{w}$  contains according
to its definition no other special vertex inside the path,  in
particular  $w$ and  $\overline{w}$ do not occur inside the path.
According to  Lemma \ref{alternace},  the path $f$ is a
palindrome.
\pfk

\pf[Proof of Theorem~\ref{PALaDefect}] It is a~direct consequence of Propositions~\ref{Pal_K}
and~\ref{H_Pal} and of Theorem~\ref{vsechno}, where the last statement is replaced with an equivalent one: There exists an integer $H$ such that for any factor $f$ of ${\mathbf u}$  with $|f| \geq H$
the longest palindromic  suffix  of $f$ is unioccurrent in $f$.
\pfk

%%%%%%%%%%%%%%%%%%%%%%%%%%%%%%%%%%%%%%%%%%%%%%%%%%%%%%%%%%%%%%%%%%%%%%%%%%%%%%%%%%%%%%%%%%%%%%%%%%%%%%%%%%%%%%%%%%
%%%%%%%%%%%%%%%%%%%%%%%%%%%%%%%%%%%%%%%%%%%%%%%%%%%%%%%%%%%%%%%%%%%%%%%%%%%%%%%%%%%%%%%%%%%%%%%%%%%%%%%%%%%%%%%%%%
%%%%%%%%%%%%%%%%%%%%%%%%%%%%%%%%%%%%%%%%%%%%%%%%%%%%%%%%%%%%%%%%%%%%%%%%%%%%%%%%%%%%%%%%%%%%%%%%%%%%%%%%%%%%%%%%%%
\section{Morphisms of class $P_{ret}$}
In this section, we will define a~new class of morphisms and we will reveal their relation with well-known morphisms of class $P$ (defined in~\cite{HoKnSi}). We will show an important role these morphisms play in
the description of words with finite defect.

\begin{de} \label{def_Pret} We say that a morphism $\varphi : \mathcal{B}^*\mapsto  \mathcal{A}^*$  is of class $P_{ret}$  if there
exists a palindrome $p \in \mathcal{A}^*$ such that
\begin{itemize}
\item \label{def_Pret_1}  $\varphi(b)p$ is a palindrome for any $b\in \mathcal{B}$,
\item \label{def_Pret_2} $\varphi(b)p$ contains exactly $2$ occurrences of $p$, one as a prefix and one as a suffix, for any
$b\in \mathcal{B}$,
 \item \label{def_Pret_3} $\varphi(b)\neq \varphi(c)$ for all $b, c \in \mathcal{B},\
b\neq c$.
\end{itemize}
\end{de}

\begin{pozn}\label{PropPret}
The following properties of the morphisms of class $P_{ret}$
are easy to prove.
\begin{enumerate}
 \item  \label{vl_Pret_1} $\varphi(w) = \varphi(v)$, where $w,v \in \mathcal{B}^*$,
implies $w= v$, i.e., $\varphi$ is injective,
  \item \label{vl_Pret_2} $\overline{\varphi(x)p} = \varphi(\overline{x})p$
  for any $x\in \mathcal{B}^*$,
  \item \label{vl_Pret_3} $\varphi(s)p$ is a palindrome if and only if $s\in \mathcal{B}^*$ is a palindrome.
\end{enumerate}
Hint for the proof of the injectivity:
If $\varphi(w) = \varphi(v)$, then $\varphi(w)p = \varphi(v)p$. This implies $w=v$ by induction on $\max\{|w|, |v|\}$: the assertion is true for $\max\{|w|, |v|\}=1$ (i.e., $|w|=|v|=1$, since the morphism $\varphi$ is not erasing from the second point of Definition~\ref{def_Pret}) from the third point
of Definition~\ref{def_Pret}; the induction is then proven using the second point of Definition~\ref{def_Pret}.
\end{pozn}

Another class of morphisms closely related to defects is {\em standard (special) morphisms of class $P$} defined in~\cite{GlJuWiZa}. We will reveal their connection with $P_{ret}$ in Section~\ref{comments}.

\begin{prop}
The class $P_{ret}$ is closed under the composition of morphisms, i.e., for any $\varphi, \sigma \in P_{ret}$ we have $\varphi \sigma \in P_{ret}$
(if the composition is well defined).
\end{prop}

\pf
Let $p_{\varphi}$ and $p_{\sigma}$ be the corresponding palindromes from the definition of $P_{ret}$ of the morphisms $\varphi$ and $\sigma$, respectively.
Then $p_{\varphi \sigma} := \varphi(p_{\sigma}) p_{\varphi}$ is a~palindrome by point {\em (3)} of Remark~\ref{PropPret} for $\varphi$.
It suffices to verify that $p_{\varphi \sigma}$ plays the role of the palindrome $p$ for the morphism $\varphi\sigma$.
\begin{itemize}
\item Take $b$ a letter. We have
$
\overline{  (\varphi \sigma) (b) p_{\varphi \sigma} }
=
\overline{ \varphi (\sigma(b) p_{\sigma} ) p_{\varphi}  }.
$
We obtain the following equalities using firstly point {\em (2)} of Remark~\ref{PropPret} for $\varphi$ and then for $\sigma$:
$$\overline{ \varphi (\sigma(b) p_{\sigma} ) p_{\varphi}  }=
\varphi ( \overline { \sigma(b) p_{\sigma}} ) p_{\varphi}
=
\varphi (\sigma(\overline{b}) p_{\sigma} ) p_{\varphi}
=
\varphi (\sigma(b) p_{\sigma} ) p_{\varphi}
=
(\varphi \sigma) (b)  p_{\varphi \sigma},
$$
i.e., $(\varphi \sigma) (b) p_{\varphi \sigma}$ is a~palindrome for all $b$.
\item Since $\varphi \in P_{ret}$, there is a~one-to-one correspondence between the occurrences of $p_{\varphi\sigma}=\varphi(p_{\sigma}) p_{\varphi}$ in $(\varphi\sigma)(b)p_{\varphi\sigma}=\varphi(\sigma(b)p_\sigma)p_\varphi$ and the occurrences of $p_\sigma$ in $\sigma(b)p_\sigma$.
    As $\sigma \in P_{ret}$, the word $\sigma(b)p_\sigma$ contains $p_\sigma$ only as a~prefix and as a~suffix. Therefore $\varphi(\sigma(b)p_\sigma)p_\varphi$
has only two occurrences of $\varphi(p_{\sigma}) p_{\varphi}$ - as a~prefix and as a~suffix.
\item
The injectivity of $\varphi$ and $\sigma$ clearly guarantees that $(\varphi\sigma)(b) \neq (\varphi\sigma)(c)$ for all $b \neq c$.
\end{itemize}

\pfk

In \cite{HoKnSi} another class of morphisms is defined.
We say that a~morphism $\varphi$ is of class $P$ if there exist a~palindrome $p$ and for every letter $a$ a~palindrome $q_a$ such that
$\varphi(a) = pq_a$.
The interest of the class $P$ has been awoken by the following question stated ibidem (however formulated in terms of dynamical systems):
``Given a~fixed point of a~primitive morphism $\varphi$ containing infinitely many palindromes, can we find a~primitive morphism $\sigma$ of class $P$
such that the factors of a~fixed point of $\sigma$ are the same?'' Let us recall that for any primitive morphism, the languages of all its fixed points are the same.
The previous question has been answered affirmatively in~\cite{BoTan} for morphisms defined on binary alphabets and in~\cite{AlBaCaDa} for periodic fixed points.

In order to reveal the relation between the classes $P$ and $P_{ret}$, we have to define the conjugation of a~morphism.
A~morphism $\sigma$ is said to be {\em conjugate} to a~morphism $\varphi$ defined on an alphabet $\mathcal A$ if there exists a~word $w \in {\mathcal A}^*$ such that
\begin{itemize}
\item
either for every letter $a \in {\mathcal A}$, the image $\varphi(a)$ has $w$ as its prefix and the image $\sigma(a)$ is obtained from $\varphi(a)$ by erasing $w$ from the beginning and adding $w$ to the end; we write
$\sigma(a)=w^{-1}\varphi(a)w$,
\item or for every letter $a \in {\mathcal A}$, the image $\varphi(a)$ has $w$ as its suffix and the image $\sigma(a)$ is obtained from $\varphi(a)$ by erasing $w$ from the end and adding $w$ to the beginning; we write
$\sigma(a)=w\varphi(a)w^{-1}$.
\end{itemize}
\begin{prop}
If $\varphi$ is a~morphism of class $P_{ret}$, then $\varphi$ is conjugate to a~morphism of class~$P$.
\end{prop}

\pf
Let $\varphi \in P_{ret}$ and let $p$ have the same meaning as in the definition of $P_{ret}$.
We will write $p = qx\overline{q}$, where $q \in \mathcal{A}^*$ and $x$ is either the empty word or a letter.
Denote by $\sigma$ a~morphism defined for all letters $a$ as $\sigma(a) = q^{-1}\varphi(a)q$.
Thus, $\varphi$ is conjugate to $\sigma$.

The word $q^{-1}\varphi(a)q$ can be written as $xy_a$ since $qx$ is a prefix of $\varphi(a)q$.
Since $\varphi(a)qx\overline{q}$ is a palindrome, $q^{-1}\varphi(a)qx\overline{q}\ {\overline{q}}^{-1} = xy_ax$ is a palindrome too.
Therefore $y_a$ is a~palindrome and $\sigma$ is of class $P$.
\pfk

The implication cannot be reversed.
Consider the alphabet $\{a,b\}$ and let $\varphi(a) = aa$ and $\varphi(b) = ab$.
It is clear that $\varphi \in P$ (for $p=a$), but $\varphi$ is not conjugate to any morphism of class $P_{ret}$ ($aaa$ is not a~complete return word of $a$).
\vspace{0.3cm}

The following theorem   shows the importance of morphisms of class
$P_{ret}$ for uniformly recurrent words with finite defect.

\begin{thm}\label{obraz} Let ${\mathbf u}\in \mathcal{A}^{\mathbb{N}}$
be a~uniformly recurrent word with finite defect.
Then there exist a~rich word ${\mathbf v} \in \mathcal{B}^{\mathbb{N}}$
and a~morphism $\varphi: {\mathcal B}^* \mapsto {\mathcal A}^*$ of class $P_{ret}$ such that

$$  {\mathbf u} = \varphi({\mathbf v}).$$
The word  ${\mathbf v}$ is
uniformly recurrent.
\end{thm}
\pf
Consider a~prefix $z$ of ${\mathbf u}$ of length $|z| > \max\{2R_{{\mathbf
 u}}(K), H\}$, where $K$ is the constant from Theorem~\ref{vsechno} and $H$ is an integer such that any factor of ${\mathbf u}$ of length $\geq H$ has its longest palindromic suffix unioccurrent. (Let us recall that the existence of $H$ is also guaranteed by Theorem~\ref{vsechno}.)
Since the language of ${\mathbf u}$ is closed under
reversal (this follows from the fact that $\mathbf u$ is uniformly recurrent
and contains infinitely many palindromes),
$\overline{z}$ is a~factor of ${\mathbf u}$ as well and its
$lps(\overline{z})$ has a~unique occurrence in $\overline{z}$.
As $|\overline{z}| > 2R_{{\mathbf u}}(K)$ any factor shorter than or equal to $K$
occurs in $\overline{z}$ at least twice. Therefore, $|lps(\overline{z})|> K$.
Hence, $lps(\overline{z})$ is a~palindromic prefix of ${\mathbf u}$ of length greater than $K$.

Denote $p:=lps(\overline{z})$. Since ${\mathbf u}$ is uniformly recurrent, the set
of return words of $p$ is finite, say  $ q_0, q_1, \ldots, q_{m-1}$ is the list of
all different return words. Let us define a~morphism $\varphi$ on the alphabet $\mathcal{B} =
\{0,1,\ldots, {m-1}\}$ by $\varphi(b) = q_b$ for all $b \in
\mathcal{B}$.  It is obvious that the morphism belongs to the
class $P_{ret}$.
Then we can write ${\mathbf u} = q_{i_0}q_{i_1}q_{i_2}\ldots $ for some sequence $(i_n)_{n\in \mathbb{N}} \in
 \mathcal{B}^{\mathbb{N}}$. Let us put ${\mathbf v} = (i_n)_{n\in
 \mathbb{N}}$.

We will  show  that any complete return word of any palindrome in the word
${\mathbf v}$ is a~palindrome as well.
According to \Cref{returnrich} this implies the richness of ${\mathbf v}$.

Let $s$ be a~palindrome in ${\mathbf v}$ and $x$ its complete return
word. Then $\varphi(x)p$ has precisely two occurrences of the
factor $\varphi(s)p$. As  $s$ is a~palindrome, $\varphi(s)p$ is a~palindrome as well
of length  $|\varphi(s)p| \geq |p|
> K$. Therefore $\varphi(x)p$ is a~complete return word of a
long enough palindrome and according to our assumption
$\varphi(x)p$ is a~palindrome as well. This together with point {\em (3)}
in Remark~\ref{PropPret} implies
$$ \varphi(x)p = \overline{\varphi(x)p} = \varphi(\overline{x})p.$$
The point {\em (2)} then  gives $x=\overline{x}$ as we claimed.\\
The uniform recurrence of ${\mathbf v}$ is obvious.
\pfk

The reverse implication does not hold, i.e.,
the set of uniformly recurrent words with finite defect is not closed under morphisms of class $P_{ret}$.
Let us provide a~construction of such a~word.

\begin{ex}
\label{ex_rich_na_nekonecno}

Let $v_0 = \epsilon$. For $i > 0$ set $v_i = \left( v_{i-1}0v_{i-1}1v_{i-1}1v_{i-1}0v_{i-1}2v_{i-1}2 \right)^{(+)}$,
where $w^{(+)}$ denotes the shortest palindrome having $w$ as a~prefix.

Note that $v_{i-1}$ is a~prefix of $v_i$ for all $i$.
Thus we can set $\displaystyle \mathbf{v} = \lim_{i \to \infty} v_i$ and $\mathbf v$ is uniformly recurrent by construction.

Denote by $\varphi$ a morphism from $P_{ret}$ defined by
$$
\varphi: \left \{ \begin{array}{l} 0 \mapsto 0100 \\ 1 \mapsto 01011 \\ 2 \mapsto 010111 \end{array} \right..
$$

As we will show in the sequel, the word $\mathbf{v}$ is rich and the defect $D(\varphi(\mathbf{v})) = \infty$.

\end{ex}

\begin{lem}
\label{v_i_rich}
For all $i$ the palindrome $v_i$ from Example~\ref{ex_rich_na_nekonecno} is rich.
\end{lem}

\pf

We will show for all $i$ that $v_i$ is rich and
$$
v_{i} = v_{i-1}0v_{i-1}1v_{i-1}1v_{i-1}0v_{i-1}2v_{i-1}2v_{i-1}0v_{i-1}1v_{i-1}1v_{i-1}0v_{i-1}.
$$
Furthermore, we will show that for all letters $x$, the word $v_ixv_i$ contains exactly $2$ occurrences of $v_i$
and $1$ occurrence of $0v_{i-1}xv_{i-1}0$.

We will proceed by induction on $i$.
For $i = 1$ and $2$ it is left up to the reader to verify the proposition.

Suppose the fact holds for $i$, $i \geq 2$.
We will show the claim for $i+1$.
Denote by $w$ the factor $$w := v_{i}0v_{i}1v_{i}1v_{i}0v_{i}2v_{i}2.$$
Note that since $v_{i}xv_{i}$ contains exactly $2$ occurrences of $v_{i}$ for all letters $x$,
the factor $w$ contains exactly $6$ occurrences of $v_{i}$.
In other words, if we find $1$ occurrence of $v_{i}$, we know all the other occurrences.

\begin{table}
\begin{tabular}{lrclc}
\# & \multicolumn{3}{c}{palindromic factors of $w$} & count \\ \hline
1 & \multicolumn{3}{c}{palindromic factors of $v_i$} & $|v_i| + 1$ \\
2 &$0v_{i-1}0,$ &$\ldots$& $,v_{i-1}0v_{i-1}0v_{i-1}$ & $|v_{i-1}|+1$ \\
3 &$1v_{i-1}0v_{i-1}1,$ &$\ldots$& $,v_{i-1}1v_{i-1}0v_{i-1}1v_{i-1}$ & $|v_{i-1}|+1$ \\
4 &$2v_{i-1}0v_{i-1}1v_{i-1}1v_{i-1}0v_{i-1}2,$ &$\ldots$& $,v_{i-1}2v_{i-1}0v_{i-1}1v_{i-1}1v_{i-1}0v_{i-1}2v_{i-1}$ & $|v_{i-1}|+1$ \\
5 &$0v_{i-1}0v_{i-1}0,$ &$\ldots$& $,v_i0v_i$ & $|v_i| - |v_{i-1}| $ \\
6 &$0v_{i-1}1v_{i-1}0,$ &$\ldots$& $,v_i1v_i$ & $|v_i| - |v_{i-1}| $ \\
7 &$0v_{i-1}2v_{i-1}0,$ &$\ldots$& $,v_i2v_i$ & $|v_i| - |v_{i-1}| $ \\
8 &$1v_i1,$ &$\ldots$& $,v_i0v_i1v_i1v_i0v_i$ & $2|v_i| + 2$ \\
9 && $2v_i2$ & & 1 \\ \hline \hline
\multicolumn{4}{l}{\textbf{total}} & $6|v_i| + 7$
\end{tabular}
\caption{Enumeration of palindromic factors of $w$.}
\label{ta_ex_pali}
\end{table}

In \Cref{ta_ex_pali} we can see the total number of palindromic factors of $w$.
Let us give a brief explanation for rows which may not be clear at first sight.
Let us recall that by the induction assumption
$$
v_{i} = v_{i-1}0v_{i-1}1v_{i-1}1v_{i-1}0v_{i-1}2v_{i-1}2v_{i-1}0v_{i-1}1v_{i-1}1v_{i-1}0v_{i-1}.
$$
Since there are exactly $11$ occurrences of $v_{i-1}$ in $v_{i}$, one can easily see that factors in rows $2$, $3$, and $4$ have not been counted in row $1$.
Rows $5$, $6$, and $7$ exploit the fact that for all letters $x$ $v_ixv_i$ contains $1$ occurrence of $0v_{i-1}xv_{i-1}0$.
One can see that the total number of palindromic factors is $6|v_i|+7=|w|+1$, therefore $w$ is rich from the definition.

As the right palindromic closure preserves the richness, we can see that $v_{i+1}$ is rich.
Moreover, since there are exactly $2$ occurrences of $v_i$ in $v_ixv_i$ for all letters $x$,
one can see that the closure will produce the following palindrome
$$
v_{i+1} = v_{i}0v_{i}1v_{i}1v_{i}0v_{i}2v_{i}2v_{i}0v_{i}1v_{i}1v_{i}0v_{i}.
$$
Take a letter $x$.
We can now rewrite $v_{i+1}xv_{i+1}$ in terms of $v_{i}$ and see
the factor $0v_{i}xv_{i}0$ occurs once and $v_{i+1}$ occurs twice
again arguing by the known count of factors $v_{i}$.

\pfk

\begin{prop}
The infinite word $\mathbf{v}$ defined in \Cref{ex_rich_na_nekonecno} is rich
and $D(\varphi(\mathbf{v})) = \infty$, where $\varphi$ is also defined in \Cref{ex_rich_na_nekonecno}.
\end{prop}

\pf

Directly from the definition of $\mathbf{v}$, one can see using the previous lemma that all its prefixes $v_i$ are rich
and therefore $\mathbf{v}$ is rich.

Denote by $p$ the palindrome from the definition of $P_{ret}$ for the substitution $\varphi$.
One can see that $p = 010$.
Take $1v_i1$, a factor of $\mathbf{v}$.
We have $\varphi(1v_i1) = 01011\varphi(v_i)p11$, a factor of $\varphi(\mathbf{v})$.
Using point {\em (3)} of Remark~\ref{PropPret}, we can see that $o_i := 1\varphi(v_i)p1$ is a~palindrome.
Now take $2v_i2$.
One can see that $\varphi(2v_i2) = 010111\varphi(v_i)p111$.
Note again the palindromic factor $o_i$.

We will now look for complete return words of $o_i$ in $\varphi(r_i)$, where
$$
r_i = 1v_i1v_i0v_i2v_i2.
$$
The word $r_i$ is clearly a~factor of $v_{i+1}$, therefore a factor of $\mathbf{v}$.
The first occurrence of $o_i$ is produced by the factor $1v_i1$ in $r_i$.
Since $\varphi$ is injective, we need to look only at occurrences of $v_i$ in $r_i$.
The next two occurrences are in the factors $1v_i0$ and $0v_i2$.
One can see that $\varphi(1v_i0) = 01011\varphi(v_i)p0$ and $\varphi(0v_i1) = 0100\varphi(v_i)p11$, i.e., the factor $o_i$
does not occur in $\varphi(r_i)$ until the factor $\varphi(2v_i2)$ occurs.
The complete return word of $o_i$ is then $O_i := 1\varphi(v_i1v_i0v_i2v_i)p1$.
By point {\em (3)} of Remark~\ref{PropPret}, as $v_i1v_i0v_i2v_i$ is not a~palindrome, neither is the complete return word $O_i$.
Therefore for each $i$ we have an oddity $\{ O_i, \overline{O_i} \}$.
According to \Cref{oddities_defect}, it implies the defect of $\varphi(\mathbf{v})$ is infinite.

\pfk

The last proposition concludes the counterexample \ref{ex_rich_na_nekonecno}.

\begin{pozn}
It is clear that the defect of an image by a morphism of class $P_{ret}$ of a word with finite defect depends on the morphism.
As \Cref{ex_rich_na_nekonecno} shows, it depends also on the original word.
To underline this fact we can take the morphism $\varphi$ from \Cref{ex_rich_na_nekonecno} and $\mathbf{u}$ the Tribonacci word, i.e., the fixed point of the Tribonacci morphism
$0 \mapsto 01$, $1 \mapsto 02$ and $2 \mapsto 0$ - a well-known rich word~\cite{DrJuPi}.
It is easy to see that $D(\varphi(\mathbf{u})) = 0$.
\end{pozn}

%%%%%%%%%%%%%%%%%%%%%%%%%%%%%%%%%%%%%%%%%%%%%%%%%%%%%%%%%%%%%%%%%%%%%%%%%%%%%%%%%%%%%%%%%%%%%%%%%%%%%%%%%%%%%%%%%%%%%%%%%%%%%%%
%%%%%%%%%%%%%%%%%%%%%%%%%%%%%%%%%%%%%%%%%%%%%%%%%%%%%%%%%%%%%%%%%%%%%%%%%%%%%%%%%%%%%%%%%%%%%%%%%%%%%%%%%%%%%%%%%%%%%%%%%%%%%%%
\section{Comments}\label{comments}

At the end of the article \cite{BlBrGaLa}, the authors  state
several open questions, among them the following one:
``Let ${\mathbf  u}$ be a~fixed point
of a primitive morphism. If the defect is finite and non-zero, is the word ${\mathbf  u}$ necessarily periodic?''

We are not able to answer this question. The following
observation is just a~small comment to it.

\begin{observation} Let ${\mathbf  u}$ be a fixed point
of a primitive morphism and let its defect $ D({\mathbf u})$ be
finite. Then there exists a rich word ${\mathbf v}$ and a morphism
$\varphi \in P_{ret}$  such that ${\mathbf u}=\varphi({\mathbf v})$ and
${\mathbf v}$ itself is a fixed point of a primitive morphism as well.
\end{observation}

\pf The rich word ${\mathbf v}$, which we have
constructed in the proof of \Cref{obraz}, is a~derived
word, as introduced by Durand in
\cite{Du}. Lemma 19 of \cite{Du} says that any derived word
of a fixed point of a primitive morphism is a fixed point of a
primitive morphism as well.

\pfk

Theorem \ref{obraz} has the form
of implication, which cannot be reversed, since Example
\ref{ex_rich_na_nekonecno} demonstrates that a morphism from
$P_{ret}$ does not preserve automatically the set of words with
finite defect. It is thus natural to ask:
\begin{description}
\item[Question 1]  Is it possible to replace the class $P_{ret}$
with a smaller one in such a way that Theorem \ref{obraz} can be
stated in the form of equivalence?
\item[Question 2]  Which morphisms from $P_{ret}$ do preserve the set of rich
words?
\item[Question 3]  How to compute $D(\varphi({\mathbf u}))$
for a rich word ${\mathbf u}$ and a morphism from $\varphi \in P_{ret}$?
\item[Question 4]  Is it feasible to characterize morphisms $\varphi$ on $\mathcal{B}^*$ with the
property that $\varphi ({\mathbf u})$ has finite defect for any
infinite word ${\mathbf u} \in \mathcal{B}^\mathbb{N}$ with finite
defect?
\end{description}
Let us comment Question 1. The authors of~\cite{GlJuWiZa} define another class of morphisms that play an important role in the study
of finite defect. They call a~morphism $\varphi$ on ${\mathcal A}^*$ a~{\em standard morphism of class $P$} (or a~{\em standard $P$-morphism}) if there exists a~palindrome $r$ (possibly empty) such that, for all $x \in {\mathcal A}, \ \varphi(x) =
rq_x$, where the $q_x$ are palindromes. If $r$ is non-empty, then some (or all) of the palindromes
$q_x$ may be empty or may even take the form $q_x=\pi_x^{-1}$
with $\pi_x$ a~proper palindromic suffix of $r$.
They say that a~standard $P$-morphism is {\em special} if:
\begin{enumerate}
\item all $\varphi(x) = rq_x$ end with different letters, and
\item whenever $\varphi(x)r = rq_xr$, with $x \in {\mathcal A}$, occurs in some $\varphi(y_1y_2\dots y_n)r$,
then this occurrence is $\varphi(y_m)r$ for some $m$ with $1 \leq m \leq n$.
\end{enumerate}
They prove the following theorem.
\begin{thm}[Theorem 6.28~\cite{GlJuWiZa}]\label{special_defect}
If $\varphi$ is a~standard special $P$-morphism on ${\mathcal A}^*$ and $\mathbf u \in {\mathcal A}^*$,
then $D({\mathbf u})<\infty$ if and only if $D(\varphi({\mathbf u}))<\infty$.
\end{thm}

However, as shown in the following proposition, standard special $P$-morphisms are not the only ones that preserve
the set of uniformly recurrent words with finite defect, thus the class of standard special morphisms is too small as an answer to Question 1. Let us add that standard special morphisms of class $P$ do not form a~subset of morphisms of class $P_{ret}$. For instance,
$\varphi: a \to aabbaabba, \ b \to ab$ is a~standard special $P$-morphism with $r=a$, but does not belong to $P_{ret}$.

\begin{prop}\label{binary}
Let $\mathbf{u}$ be a~binary uniformly recurrent word such that $D(\mathbf{u})$ is finite.
Let $\varphi$ be a~morphism of class $P_{ret}$.
Then $D\left( \varphi(\mathbf{u}) \right)$ is finite.
\end{prop}
\begin{lem}\label{conjugate_to_special}
Let $\varphi$ be a morphism of class $P_{ret}$ on $\{0,1\}^*$.
Then $\varphi$ is conjugate to a~standard special $P$-morphism.
\end{lem}
\begin{proof}
Let $p$ be the palindrome corresponding to $\varphi$ in the definition of $P_{ret}$.
Denote by $p_1$ the longest common suffix of $\varphi(0)$ and $\varphi(1)$.
Denote by $p_2$ a word such that $pp_2$ is the longest common prefix of $\varphi(0)p$ and $\varphi(1)p$.
Using properties of $P_{ret}$ we have $p_1 = \overline{p_2}$.
Define $\sigma(0)=p_1\varphi(0)p_1^{-1}$ and $\sigma(1)=p_1\varphi(1)p_1^{-1}$. Then $\varphi$ is conjugate to $\sigma$
and $\sigma$ is a~standard special $P$-morphism with the corresponding palindrome $r=p_1pp_1$.
\end{proof}
\begin{proof}[Proof of Proposition~\ref{binary}]
By Lemma~\ref{conjugate_to_special} the morphism $\varphi$ is conjugate to a~standard special $P$-morphism $\sigma$.
Clearly, the languages of $\varphi({\mathbf u})$ and $\sigma({\mathbf u})$ are the same, hence $D\left( \varphi(\mathbf{u}) \right)=D\left( \sigma(\mathbf{u}) \right)$.
Theorem~\ref{special_defect} implies that $D(\sigma({\mathbf u}))<\infty$.
\end{proof}

%%%%%%%%%%%%%%%%%%%%%%%%%%%%%%%%%%%%%%%%%%%%%%%%%%%%%%%%%%%%%%%%%%%%%%%%%%%%%%%%%%%%%%%%%%%%%%%%%%%%%%%%%%%%%%%%%%%%%%%%%%%%%%%
%%%%%%%%%%%%%%%%%%%%%%%%%%%%%%%%%%%%%%%%%%%%%%%%%%%%%%%%%%%%%%%%%%%%%%%%%%%%%%%%%%%%%%%%%%%%%%%%%%%%%%%%%%%%%%%%%%%%%%%%%%%%%%%
\section{Acknowledgements}
We acknowledge financial support by the Czech Science Foundation grant 201/09/0584, by
the grants MSM6840770039 and LC06002 of the Ministry of Education, Youth, and Sports of the Czech Republic,
and by the grant SGS10/085OHK4/1T/14 of the Grant Agency of the Czech Technical University in
Prague.
%%%%%%%%%%%%%%%%%%%%%%%%%%%%%%%%%%%%%%%%%%%%%%%%%%%%%%%%%%%%%%%%%%%%%%%%%%%%%%%%%%%%%%%%%%%%%%%%%%%%%%%%%%%%%%%%%%%%%%%%%%%%%%%
%%%%%%%%%%%%%%%%%%%%%%%%%%%%%%%%%%%%%%%%%%%%%%%%%%%%%%%%%%%%%%%%%%%%%%%%%%%%%%%%%%%%%%%%%%%%%%%%%%%%%%%%%%%%%%%%%%%%%%%%%%%%%%%

\end{document}